\documentclass[a4paper,12pt]{article}
\usepackage{amsmath,amsthm,amssymb,graphicx,subfigure}
\usepackage{bbm}
\usepackage{graphicx,verbatim}
\usepackage{color}
\usepackage [english]{babel}
\usepackage [latin1] {inputenc}
\usepackage [T1] {fontenc}
\usepackage{amssymb}
\usepackage{amsfonts}
\usepackage{dsfont}
\usepackage{amsmath}
\usepackage{stmaryrd}
\usepackage{graphicx}
\newtheorem{theorem}{Theorem}[section]

\newtheorem{proposition}[theorem]{Proposition}

\date{}
\title{Exponential moments of the argument of the Riemann zeta function on the critical line }
\author{
        Joseph \textsc{Najnudel}     \footnote{\texttt{joseph.najnudel@bristol.ac.uk}}
       } 

\begin{document}
\maketitle
\begin{abstract}
In this article, we give, under the Riemann hypothesis, an upper bound for the exponential moments of the imaginary part of the logarithm of the Riemann zeta function on the critical line. Our result, which gives information on the fluctuations of the distribution of the zeros of $\zeta$, has the same accuracy as the result obtained by Soundararajan in \cite{Sound} for the moments of $|\zeta|$. 

\end{abstract}

\section{Introduction}
The behavior of the Riemann zeta function on the critical line has been intensively studied, in particular in relation with the Riemann hypothesis. A natural question concerns the order of magnitude of the moments of $\zeta$: 
$$\mu_k(T) := \frac{1}{T} \int_0^T |\zeta (1/2 + it)|^{2k} dt = \mathbb{E} [ |\zeta (1/2 + iUT)|^{2k}],$$
where $k$ is a positive real number and $U$ is a uniform random variable in $[0,1]$. 
It is believed that the order of magnitude of $\mu_k(T)$ is $(\log T)^{k^2}$ for fixed $k$ and $T$ tending to 
infinity. More precisely, it is conjectured that there exists $C_k > 0$ such that 
$$\mu_k(T) \sim_{T \rightarrow \infty} C_k (\log T)^{k^2}.$$
An explicit expression of $C_k$ has been predicting by Keating and Snaith \cite{bib:KSn}, using an expected analogy between $\zeta$ and the characteristic polynomials of random unitary matrices. The conjecture has been only proven for $k =1$ by Hardy and Littlewood and for $k = 2$ by Ingham (see Chapter VII of \cite{Tit}). 

The weaker conjecture $\mu_k(T) = T^{o(1)}$ for fixed $k > 0$ and $T \rightarrow \infty$ is equivalent to the Lindel\"of hypothesis which states that $|\zeta(1/2 + it) | \leq t^{o(1)}$ when $t$ goes to infinity. The Lindel\"of hypothesis is a still open conjecture which can be deduced from the Riemann hypothesis. 

Under the Riemann hypothesis, it is known that $(\log T)^{k^2}$ is the right order of magnitude for 
$\mu_k(T)$. In \cite{Sound}, Soundararajan proves that 
$$\mu_k(T) = (\log T)^{k^2 + o(1)}$$
for fixed $k > 0$ and $T$ tending to infinity. 
In \cite{Harper}, Harper improves this result by showing that 
$$\mu_k(T) \ll_k (\log T)^{k^2},$$
for all $k > 0$ and $T$ large enough, the notation 
$A \ll_x B$ meaning that there exists $C > 0$ depending only on $x$ such that $|A| \leq C B$. 
In \cite{HRS19}, Heap, Radziwi\l\l \, and Soundararajan prove this bound unconditionally, when $0 \leq k \leq 2$. 

On the other hand, the lower bound 
$$\mu_k(T) \gg_k (\log T)^{k^2}$$
has been proven by Ramachandra (\cite{Ra78, Ra95}) and Heath-Brown \cite{HB81}, assuming the Riemann hypothesis, and, for $k \geq 1$,  by Radziwi\l\l \, and Soundararajan \cite{RS13}, unconditionally. Other lower bounds on moments of $L$-functions have also been found by Rudnick and Soundararajan \cite{RS05}. 

The moment $\mu_k(T)$ can be written as follows: 

$$\mu_k(T) = \mathbb{E} [ \exp ( 2k \Re \log \zeta(1/2 + iUT))].$$
Here $\log \zeta$ denotes the unique determination of the logarithm which is well-defined and countinous everywhere 
except at the left of the zeros and the pole of $\zeta$, and which is real on the interval $(1, \infty)$. 

It is now natural to also look at similar moments written in terms of the imaginary part of $\log \zeta$: 
$$\nu_k (T) =  \mathbb{E} [ \exp ( 2k \Im \log \zeta(1/2 + iUT))].$$
Note that $\Im \log \zeta$ is directly related to the fluctuations of the distribution of the zeros of $\zeta$ with respect to their "expected distribution": we have
$$N(t) = \frac{t}{2 \pi} \log  \frac{t}{2 \pi e} + \frac{1}{\pi} \Im \log \zeta(1/2 + it) + \mathcal{O}(1),$$
where $N(t)$ is the number of zeros of $\zeta$ with imaginary part  between $0$ and $t$. 

In the present article, we prove, conditionally on the Riemann hypothesis, an upper bound on $\nu_k(T)$ with the same accuracy as the upper bound on $\mu_k(T)$ obtained by Soundararajan in \cite{Sound}.

The general strategy is similar, by integrating estimates on the tail of the distribution of $\Im \log \zeta$, obtained 
by using bounds on moments of sums on primes coming from the logarithm of the Euler product of $\zeta$. The main difference with the paper by Soundararajan \cite{Sound} is that we do  not directly use an upper bound of $\Im \log \zeta$ which is similar to the upper bound of $\log |\zeta|$ given in his Proposition. Instead, in order to
estimate $\Im \log \zeta$ in terms of sums on primes, we intensively use the fact that 
$\Im \log \zeta(1/2 + it)$ cannot decrease too fast when $t$ increases, because of the link between $\Im \log \zeta$ and the distribution of the zeros of $\zeta$. 
In Titchmarsh \cite{Tit}, Theorem 14.21, one gets an estimate which is similar to the upper bound of the Proposition in \cite{Sound}. It may be possible to use this estimate in order to get another proof of our main result, more similar to the proof given in \cite{Sound}. 

We expect that such bound is not optimal and that the exact order of magnitude of the moment is $(\log T)^{k^2}$, which is obtained when we do the Gaussian approximation of $\Im \log \zeta$ corresponding to Selberg's central limit theorem. It would be interesting to know if the strategy of Harper \cite{Harper} can be adapted to $ \Im \log \zeta$, in order to get a sharp upper bound on its exponential moments: Titchmarsh \cite{Tit}, Theorem 14.21 may be useful for this purpose. We do not know if the techniques used in order to get lower bounds of moments of $|\zeta|$ can be adapted to the exponential moments of $\Im \log \zeta$: such adaptation, if it is possible, is not straightforward, since the usual Dirichlet series of $\zeta$ cannot be directly used in order to estimate $\Im \log \zeta$, because this value is determined by the value of $\zeta$ only modulo $2 \pi$, whereas $\Re \log \zeta$ is completely determined by $\zeta$.

The precise statement of our main result is the following: 
\begin{theorem}
Under the Riemann hypothesis, for all $k \in \mathbb{R}$, $\varepsilon> 0$, 
$$\mathbb{E} [\exp(2 k \Im \log \zeta(1/2 + iTU))] \ll_{k, \varepsilon} (\log T)^{k^2 + \varepsilon},$$
where $U$ is a uniform variable on $[0,1]$. 
\end{theorem}
The proof of this result is divided into two main parts. In the first part, we bound the tail of the distribution of $\Im \log \zeta(1/2 + iTU)$ in terms of the tail of an averaged version of this random variable. In the second part, we 
show that this averaged version is close to a sum on primes, whose tail is estimated from bounds on its moments. Combining this estimate with the results of the first part gives a proof of the main theorem. 

\section{Comparison of $\Im \log \zeta$ with an averaged version}

The imaginary part of $\log \zeta$ varies in a smooth and well-controlled way on the critical line when there are no zeros, and has positive jumps of $\pi$ when there is a zero. We deduce that it cannot decrease too fast. 
More precisely, the following holds:

\begin{proposition}
For $2 \leq t_1 \leq t_2$, we have 
$$\Im \log \zeta(1/2 + it_2) \geq \Im \log \zeta(1/2 + i t_1) - (t_2 - t_1) \log t_2 + \mathcal{O}(1). $$
\end{proposition}
\begin{proof}
Is is an easy consequence of Theorem 9.3 of Titchmarsh \cite{Tit}: for example, see Proposition 4.1 of \cite{bib:Naj} for details. 
\end{proof}

We will now define some averaging of $\Im \log \zeta$ around points of the critical line. From the previous proposition, if  
$\Im \log \zeta$ is large at some point $1/2 + it_0$ of the critical line, then it remains large on some segment $[1/2 + it_0,1/2 + i(t_0 + \delta)]$ which tends to also give a large value of an average of $\Im \log \zeta(1/2 + it)$ for $t$ around $t_0$. Our precise way of averaging is the following. 
We fix a function $\varphi$ satisfying the following properties: $\varphi$ is real, nonnegative, even, dominated by any negative power at infinity, and its Fourier transform is compactly supported, takes values in $[0,1]$, is even and equal to $1$ at zero. The Fourier transform is normalized as follows: 
$$\widehat{\varphi}(\lambda) = \int_{-\infty}^{\infty} \varphi(x) e^{-i \lambda x} dx.$$
For $H> 0$ we define an averaged version of $\Im \log \zeta$  as follows: 
$$I(\tau,H) :=  \int_{-\infty}^{\infty} 
\Im \log \zeta (1/2 +  i (\tau + t H^{-1}))\varphi (t) dt.$$
The following result holds:
\begin{proposition}
Let $\varepsilon \in (0,1/2)$. Then, there exist $a, K > 1$, depending only on $\varphi$ and $\varepsilon$, and satisfying the following property. 
For $T > 100$, $\tau \in [\sqrt{T},T]$, $ K < V < \log T$, $H := K V^{-1} \log T $, the inequalities 
$$\Im \log \zeta(1/2 + i \tau)  \geq V,$$
$$\Im \log \zeta(1/2 + i (\tau - e^r H^{-1}))  \geq - 2  V$$
for all integers $r$ between $0$ and $\log \log T$, 
together imply 
$$I (\tau + a H^{-1}, H) \geq (1-\varepsilon) V.$$ 
Similarly, the inequalities 
$$\Im \log \zeta(1/2 + i \tau)  \leq -V,$$
$$\Im \log \zeta(1/2 + i (\tau + e^r H^{-1}))  \leq 2 V$$
for all integers $r$ between $0$ and $\log \log T$, 
together imply 
$$I (\tau - a H^{-1}, H) \leq -(1-\varepsilon) V.$$ 
\end{proposition}
\begin{proof}
First, we observe that $H > 1$ since $K > 1$ and $V < \log T$. We deduce that for all the values of $s$ such that 
$\Im \log \zeta (1/2 + is)$ is explicitly written in the proposition, $ \sqrt{T} - \log T \leq s \leq T + \log T$: in particular 
$s > 2$ since $T > 100$, and we can apply the previous proposition to compare these values of $\Im \log \zeta$.

If $\Im \log \zeta (1/2 + i \tau) \geq V$, then 
for all $t \geq 0$, 
\begin{align*}
& \Im \log \zeta (1/2 + i (\tau + t H^{-1}) ) \geq V - t H^{-1} \log (\tau + tH^{-1}) 
+ \mathcal{O}(1) \\ & 
\geq V - t K^{-1} V (\log T)^{-1} \log (\tau +  t H^{-1}) + \mathcal{O}(1).
\end{align*}
Since $H > 1$ and $\tau \leq T$, 
$$\Im \log \zeta (1/2 + i (\tau + t H^{-1}) )
\geq V ( 1 - t K^{-1} (\log T)^{-1} \log(T + t)) + \mathcal{O}(1).$$
We have 
$$\log (T + t) = \log T + \log ( 1 + t/T) \leq \log T + \log (1+t),$$
and then, integrating against $\varphi(t-a)$ from $0$ to $\infty$, 
\begin{align*}
& \int_0^{\infty} \Im \log \zeta (1/2 + i (\tau + t H^{-1}) ) \varphi(t-a) dt
\\ & \geq V  \left[ \int_{0}^{\infty} \varphi(t-a) dt  -   K^{-1} 
\int_{0}^{\infty} t \varphi(t-a) dt   \right. \\ & \left. - K^{-1}  (\log T)^{-1}  \int_{0}^{\infty} t \log (1+t) \varphi(t-a) dt \right] + 
\mathcal{O}(1),
\end{align*}
Since $\varphi$ is integrable against $t$ and $t \log (1+t)$ (it is rapidly decaying at infinity), since $ V K^{-1} > 1$ and then $\mathcal{O}(1) = \mathcal{O}( V K^{-1})$, 
and since the integral of $\varphi$ on $\mathbb{R}$ is $\widehat{\varphi}(0) = 1$, we get
$$\int_0^{\infty} \Im \log \zeta (1/2 + i (\tau + t H^{-1}) ) \varphi(t-a) dt
\geq V ( 1 -  \mathcal{O}_{a, \varphi}( K^{-1})).$$
We deduce 
$$\int_0^{\infty} \Im \log \zeta (1/2 + i (\tau + t H^{-1}) ) \varphi(t-a) dt
\geq V (1 - \varepsilon/2),$$
when $K$ is large enough, depending on $a$, $\varepsilon$ and $\varphi$. 
Now, let us consider the same integral between $-\infty$ and $0$. 
For $0 \leq r \leq \log \log T$ integer and $u \in [0, e^{r} - e^{r-1}]$ for $r \geq 1$, $u \in [0,1]$ for $r = 0$, 
$$\Im \log \zeta (1/2 + i (\tau - (e^r - u) H^{-1})) 
- \Im \log \zeta (1/2 + i (\tau - e^r  H^{-1})) \geq  - u H^{-1} \log T + \mathcal{O}(1),$$
and then 
$$\Im \log \zeta (1/2 + i (\tau - (e^r - u) H^{-1})) 
\geq - 2 V - u K^{-1} V + \mathcal{O}(1)$$
Since $u \leq e^r$, $K > 1$, and $ 1 + e^r - u \geq e^{r-1}$, 
$$\Im \log \zeta (1/2 + i (\tau - (e^r - u) H^{-1})) 
\geq - 2  V  - e (1 + e^r - u) V + \mathcal{O}(1),$$
and then, for all $t \in [-  e^{\lfloor \log \log T \rfloor}, 0]$, 
$$\Im \log \zeta (1/2 + i (\tau + t H^{-1}))  
\geq - \mathcal{O}( V (1+|t|)),$$
after taking into account the fact that $V > 1$. 
If  $K$  is  large enough depending on $\varphi$, this estimate remains true for $t <  -  e^{\lfloor \log \log T \rfloor}$, since by Titchmarsh \cite{Tit}, Theorem 9.4, and by the fact that $|t| \geq e^{\log \log T -1 } \gg  \log T$,
\begin{align*}
|\Im \log \zeta (1/2 + i (\tau + t H^{-1}))   | &  \ll
\log ( 2 + \tau + |t| H^{-1})
  \ll \log (T + |t|) \\ & = \log T + \log ( 1 + |t|/T)  \leq \log T + |t|/T \ll |t|,
\end{align*}
whereas $V > K > 1$. 
Integrating against $\varphi(t-a)$, we get 
$$\int_{-\infty}^0  \Im \log \zeta (1/2 + i (\tau + t H^{-1}) ) \varphi(t-a) dt
\geq - \mathcal{O} ( I V ),$$
where $$I = \int_{-\infty}^0 (1 + |t|)  \varphi(t-a) dt
\leq  \int_{-\infty}^{-a} (1 + |s|)  \varphi(s) ds \underset{a \rightarrow \infty}{\longrightarrow} 0.$$
Hence, for $a$ large enough depending on $\varepsilon$ and $\varphi$, 
$$\int_{-\infty}^0  \Im \log \zeta (1/2 + i (\tau + t H^{-1}) ) \varphi(t-a) dt
\geq -  \varepsilon V/2.$$
Adding this integral to the same integral on $[0, \infty)$, we deduce the first part of the proposition.
The second part is proven in the same way, up to minor modifications which are left to the reader. 
\end{proof}
In the previous proposition, if we take $\tau$ random and uniformly distributed in $[0,T]$, we deduce the following result: 
\begin{proposition}
For  $\varepsilon \in (0,1/2)$, $K$ as in the previous proposition (depending on $\varepsilon$ and $\varphi$), $T > 100$, $K < V < \log T$, 
$H = KV^{-1} \log T$, $U$ uniformly distributed on $[0,1]$, 
\begin{align*}
\mathbb{P} 
[ | & \Im \log \zeta (1/2 + i UT)|  \geq V]
 \leq \mathbb{P} [ |I(UT, H)| \geq (1-\varepsilon) V ] 
 \\ & + ( 1 + \log \log T) \mathbb{P} 
[ |\Im \log \zeta (1/2 + i UT)| \geq 2 V] + \mathcal{O}_{\varepsilon,\varphi} (T^{-1/2}).
\end{align*} 
\end{proposition}

\begin{proof}
We have immediately, by taking $\tau = UT$,
\begin{align*}
\mathbb{P} 
[ |\Im \log \zeta (1/2 + i UT)| & \geq V]
 \leq \mathbb{P} [ I(UT + a H^{-1}, H) \geq (1-\varepsilon) V ] 
\\ & +  \mathbb{P} [ I(UT - a H^{-1}, H) \leq -(1-\varepsilon) V ] 
 \\ & + \sum_{r = 0}^{ \lfloor \log \log T \rfloor} \mathbb{P} 
[ \Im \log \zeta (1/2 + i (UT + e^r H^{-1}) ) \geq 2 V] 
\\ & + \sum_{r = 0}^{ \lfloor \log \log T \rfloor} \mathbb{P} 
[ \Im \log \zeta (1/2 + i (UT - e^r H^{-1}) ) \leq -  2 V]  
\\ & + \mathcal{O}(T^{-1/2}),
\end{align*} 
the last term being used to discard the event $UT \leq \sqrt{T}$. 
Now, for $u \in \mathbb{R} $, the symmetric difference between the uniform laws on $[0,T]$ and $[u H^{-1}, T + u H^{-1}]$ 
is dominated by a measure of total mass $\mathcal{O} ( |u| H^{-1} T^{-1})$. Hence, 
in the previous expression, we can replace $UT + u H^{-1}$ by $UT$ in each event, with the cost of an error term 
 $\mathcal{O} ( |u| H^{-1} T^{-1}) = \mathcal{O}(|u| T^{-1})$. The values of $|u|$ which are involved are less than 
$\max(a, \log T)$, and there are $\mathcal{O}(\log \log T)$ of them. 
Hence, we get an error term $\mathcal{O} (T^{-1}(a + \log T) \log \log T)  = \mathcal{O}_{\varepsilon, \varphi} (T^{-1/2})$ since $a$ depends only on $\varepsilon$ and $\varphi$. 
\end{proof}
We can now iterate the proposition: applying it for $V, 2V, 4V,...$. After a few manipulations, it gives the following:
\begin{proposition} \label{sumV}
For $\varepsilon \in (0,1/2)$, $K$ as in the previous proposition (depending on $\varepsilon$ and $\varphi$), $T > 100$, $K< V < \log T$, 
$H = KV^{-1} \log T$, $U$ uniformly distributed on $[0,1]$, 
\begin{align*}
\mathbb{P} 
[ | & \Im \log \zeta (1/2 + i UT)|  \geq V]
   \\ &  \leq \sum_{r = 0}^{p-1} (1 +\log \log T)^r  \mathbb{P} [ |I(UT, 2^{-r} H)| \geq (1-\varepsilon) 2^r V ] +   \mathcal{O}_{\varepsilon,\varphi} (T^{-1/3}),
\end{align*} 
where $p$ is the first integer such that $2^p V \geq \log T$. 
\end{proposition}
\begin{proof}
We iterate the formula until the value of $V$ reaches $\log T$. The number of steps is dominated by $\log \log T - \log K \leq \log \log T$.  Each step gives an error term 
of at most $\mathcal{O}_{\varepsilon, \varphi} ( ( 1+ \log \log T)^{\mathcal{O} (\log \log T)}  T^{-1/2})$. Hence, the total error is $  \mathcal{O}_{\varepsilon,\varphi} (T^{-1/3})$.
We deduce
\begin{align*}
\mathbb{P}  
[ |\Im & \log \zeta (1/2 + i UT)|  \geq V] \leq \sum_{r = 0}^{p-1} (1 +\log \log T)^r  \mathbb{P} [ |I(UT, 2^{-r} H)| \geq (1-\varepsilon) 2^r V ] 
\\ & + (1 +\log \log T)^p \, \mathbb{P} 
[ |  \Im \log \zeta (1/2 + i UT)|  \geq  2^p V] + \mathcal{O}_{\varepsilon, \varphi} (T^{-1/3}),
\end{align*}
Under the Riemann hypothesis, Theorem 14.13 of Titchmarsh \cite{Tit} shows that $|  \Im \log \zeta (1/2 + i UT)|   
\ll (\log \log T)^{-1} \log T$ under the Riemann hypothesis. Hence the probability that $|\Im \log \zeta|$ is larger than $2^p V \geq \log T$ is equal to zero if $T$ is large enough, which can be assumed (for small $T$, we can absorb everything in the error term). 
\end{proof}
\section{Tail distribution of the averaged version of  $\Im \log \zeta$ and proof of the main theorem}
The averaged version $I(\tau, H)$ of $\Im \log \zeta$ can be written in terms of  sums indexed by primes: 

\begin{proposition} \label{Naj32}
Let us assume the Riemann hypothesis. There exists $\alpha > 0$, depending only on the function $\varphi$, such that 
for all $\tau \in \mathbb{R}$, $0 < H < \alpha \log  (2+|\tau|)$, $$ I (\tau, H)
 = \Im \sum_{p \in \mathcal{P}} p^{-1/2 - i \tau} \widehat{\varphi} (H^{-1} \log p) 
 + 
\frac{1}{2} \Im \sum_{p \in \mathcal{P}} p^{-1- 2 i \tau} \widehat{\varphi} (2 H^{-1} \log p) + 
\mathcal{O}_{\varphi} (1),
$$
$\mathcal{P}$ being the set of primes. 
\end{proposition} 
\begin{proof}
This result is an immediate consequence of Proposition 3.2 of \cite{bib:Naj}, which is itself deduced from Lemma 5 of 
Tsang \cite{Tsang86}.
\end{proof}
We will now estimate the tail distribution of $I(UT, H)$, where $U$ is uniformly distributed on $[0,1]$, by using upper bounds of the moments of the sums on primes involved in the previous proposition. 
 We use Lemma 
3 of  Soundararajan \cite{Sound}, which is  presented as a standard mean value estimate by the author (a similar result can be found in Lemma 3.3 of Tsang's thesis \cite{Tsang84}), and which can be stated as follows: 
\begin{proposition} \label{Lemma3Sound}
For $T$ large enough and $2 \leq x \leq T$, for $k$ a natural number such that 
$x^k \leq T/\log T$, and for any complex numbers $a(p)$ indexed by the primes, 
we have 
$$\int_{T}^{2T} \left| \sum_{p \leq x, p \in \mathcal{P}} \frac{a(p)}{p^{1/2 + it}} \right|^{2k} dt
\ll T k! \left( \sum_{p \leq x, p \in \mathcal{P}} \frac{|a(p)|^2}{p} \right)^k.$$
\end{proposition} 
From Propositions \ref{Naj32} and \ref{Lemma3Sound}, we can deduce the following tail estimate: 

\begin{proposition}
For $T$ large enough,  $\varepsilon \in (0,1/10)$, $K > 1$ depending only on $\varepsilon$ and $\varphi$, $0 < V < \log T$, 
$H = K V^{-1} \log T$, we have 
$$\mathbb{P} [ |I(UT,  H)| \geq (1-\varepsilon) V] \ll_{\varepsilon, \varphi} e^{- (1- 3 \varepsilon) V^2 / \log \log T}
 + e^{-  b_{\varepsilon, \varphi} V \log V} + T^{-1/2},$$
where $b_{\varepsilon, \varphi} > 0$ depends only on $\varepsilon$ and $\varphi$. 
\end{proposition}
\begin{proof}
We can assume $V \geq 10 \sqrt{ \log \log T}$ and $V$, $T$ are larger than any given quantity depending only on $\varepsilon$ and $\varphi$: otherwise the upper bound is trivial. 
In particular, if we choose $\alpha$ as in Proposition \ref{Naj32}, it depends only on $\varphi$, and then, for $K > 1$ depending only on $\varepsilon$ and $\varphi$, we can assume $V > 2K/\alpha$. 
From this inequality, we deduce $H  < ( \alpha/2) \log T$, which gives 
$H < \alpha \log (2 + UT)$ with probability $  1 - \mathcal{O} (T^{-1/2})$. Under this condition, 
 Proposition \ref{Naj32} applies to $\tau = UT$, and we deduce: 
$$I(UT,  H) = \Im S_1 + \Im S_2 + \Im S_3 + \mathcal{O}_{\varphi}(1),$$
where 
$$S_1 := \sum_{p \in \mathcal{P}, p \leq T^{1/( V \log \log T)} } p^{-1/2 - i UT} \widehat{\varphi} (H^{-1} \log p),$$
$$S_2 := \sum_{p \in \mathcal{P}, p > T^{1/( V \log \log T)} } p^{-1/2 - i UT} \widehat{\varphi} (H^{-1} \log p),$$
$$S_3 := \frac{1}{2} \sum_{p \in \mathcal{P} } p^{-1 - 2 i UT} \widehat{\varphi} (2 H^{-1} \log p).$$
Since we can assume $V$ large depending on $\varepsilon$ and $\varphi$, we can suppose that 
the term $ \mathcal{O}_{\varphi}(1)$ is smaller than $\varepsilon V/20$: 
$$|I(UT,  H)| \leq |S_1 |+ |S_2| + |S_3 |+ \varepsilon V/20,$$
with probability $  1 - \mathcal{O} (T^{-1/2})$.
We deduce 
$$ \mathbb{P} [ |I(UT,  H)| \geq (1-\varepsilon) V] \leq \mathbb{P} [ |S_1| \geq (1- 1.1 \varepsilon) V] $$ $$+ 
\mathbb{P} [| S_2 |\geq \varepsilon V/100] + \mathbb{P} [| S_3 | \geq \varepsilon V/100]
+ \mathcal{O}(T^{-1/2}).$$

We estimate the tail of these sums by applying Markov inequality to their moment of order $2k$, $k$ being a suitably chosen integer.  Since $\widehat{\varphi}$ is compactly supported, all the sums have finitely many non-zero terms, and 
we can apply, for $T$ large, the lemma to all values of $k$ up to $\gg_{\varphi} \log (T /\log T) / H$, 
i.e. $\gg_\varphi V/K$, and then $\gg_{\varepsilon, \varphi} V$.  
For the sum $S_1$, we can even go up to $(V \log \log T)/2$. 
For $S_2$, we can take $k = \lfloor c_{\varepsilon, \varphi} V \rfloor$, for a suitable $c_{\varepsilon, \varphi} > 0$ depending only on $\varepsilon$ and $\varphi$. The moment of order $2k$ is 
$$\ll k! \left(\sum_{T^{1/(V \log \log T)} < p \leq  e^{\mathcal{O}_\varphi (H)}} p^{-1} \right)^k
\leq  k^k ( \log \log \log T  + \mathcal{O}_{\varepsilon, \varphi}(1) )^k$$
Hence, 
$$\mathbb{P} [ | S_2 |\geq \varepsilon V/100] \leq (\varepsilon V/100)^{- 2 \lfloor c_{\varepsilon, \varphi} V \rfloor} 
( c_{\varepsilon, \varphi} V  ( \log \log \log T  + \mathcal{O}_{\varepsilon, \varphi}(1)  ) )^{ \lfloor c_{\varepsilon, \varphi} V \rfloor }
$$
Since we have assumed $V \geq 10 \sqrt{\log \log T}$, 
we have
$$(\varepsilon V/100)^{-2}    c_{\varepsilon, \varphi} V  ( \log \log \log T  + \mathcal{O}_{\varepsilon, \varphi}(1)  ) 
\leq V^{-0.99}$$
and 
$$ \lfloor c_{\varepsilon, \varphi} V \rfloor \geq 0.99  \, c_{\varepsilon, \varphi} V,$$
for $T$ large enough depending on $\varepsilon$ and $\varphi$. 
Hence 
$$ \mathbb{P} [ | S_2| \geq \varepsilon V/100] \leq V^{-0.98 c_{\varepsilon, \varphi} V},$$
which is acceptable. 
An exactly similar proof is available for $S_3$, since we even get a $2k$-th moment bounded by $k! (\mathcal{O}(1))^k$. 

For $S_1$, the $2k$-th moment is 
$$\ll k! (\log \log T + \mathcal{O}_{\varepsilon, \varphi}(1))^k$$
for $k \leq (V \log \log T)/2$. Hence, the probability that $|S_1| \geq W := (1 - 1.1 \varepsilon)V$ is 
$$\ll W^{-2k} k! (\log \log T + \mathcal{O}_{\varepsilon, \varphi}(1))^k
$$
We appoximately optimize this expression in $k$. If $V \leq (\log \log T)^2/2$, we can take 
$k = \lfloor W^2/ \log \log T \rfloor$ since this expression is smaller than $V \log \log T/2$.  Notice that since
 $V \geq 10 \sqrt{\log \log T}$ and $\varepsilon < 1/10$, we have $W \geq  8 \sqrt{\log \log T}$ and $k$ is strictly positive. 
The probability that  $|S_1| \geq W$ is then 
$$ \ll [W^{-2} (k/e) (\log \log T +  \mathcal{O}_{\varepsilon, \varphi}(1))]^k \sqrt{k} $$
The quantity inside the bracket is smaller than $e^{-(1- (\varepsilon/100))}$ for $T$ large enough depending on $\varepsilon$ and $\varphi$. Hence, in this case, the probability is 
$$ \leq e^{- (1- (\varepsilon/100)) k} \sqrt{k} \ll_{\varepsilon} e^{- (1- (\varepsilon/50)) k}
 \ll e^{ - (1- (\varepsilon/50)) ( 1- 1.1 \varepsilon)^2 V^2 / \log \log T}.$$
This is acceptable. 
If $V > (\log \log T)^2/2$, we take $k = \lfloor V \log \log T/2 \rfloor$. 
We again get a probability
$$ \ll  [W^{-2} (k/e) (\log \log T +  \mathcal{O}_{\varepsilon, \varphi}(1))]^k \sqrt{k}.$$
Inside the bracket, the quantity is bounded, for $T$ large enough depending on $\varepsilon$ and $\varphi$, 
 by 
\begin{align*}
& W^{-2} (V \log \log T/2e) (1.001 \log \log T) \leq W^{-2} V   (\log \log T)^2/5.4  \\ & = V^{-1} (\log \log T)^2 (1-1.1\varepsilon)^{-2}/5.4   \leq 2  (1-1.1\varepsilon)^{-2}/5.4 \leq 1/2
\end{align*}
Hence, we get  a probability
$$\ll 2^{-k} \sqrt{k} \ll e^{-k/2}  \ll e^{- V \log \log T/4} \ll e^{-V \log V/4},$$
the last inequality coming from the fact that $V < \log T$ by assumption. This is again acceptable. 
\end{proof}
We then get the following bounds for the tail of $\Im \log \zeta$, which easily imply the main theorem by integrating
against $e^{2 k V}$: 
\begin{proposition}
For all $\varepsilon \in (0,1/10)$, $V > 0$, 
$$\mathbb{P} [ | \Im \log \zeta(1/2 + iUT)| \geq V] 
\ll_{\varepsilon}   e^{(\log \log \log T)^3} e^{- (1-\varepsilon) V^2/ \log \log T} + e^{- c_{\varepsilon} V \log V},$$
where $c_{\varepsilon} > 0$ depends only on $\varepsilon$. 
\end{proposition}
\begin{proof}
We fix a function $\varphi$ satisfying the assumptions given at the beginning: this function will be considered as universal, and then we will drop all the dependences on $\varphi$ in this proof. 
From Theorem 14.13 of Titchmarsh \cite{Tit}, we can assume $V \ll (\log \log T)^{-1} \log T$ and then $V < \log T$ for $T$ large (otherwise the probability is zero). 
We can then  also assume $T$ larger than any given quantity depending only on $\varepsilon$ (if $T$ is small, $V$ is small), and $V \geq 10 \sqrt{\log \log T}$. Under these assumptions, we can suppose $V > K$  if $K > 0$ depends only on $\varepsilon$, which allows to apply Proposition \ref{sumV}. 
The error term $\mathcal{O}_{\varepsilon}(T^{-1/3})$ can be absorbed in $\mathcal{O}_{\varepsilon}( e^{-c_{\varepsilon} V \log V})$ since $V \ll (\log \log T)^{-1} \log T$. The sum in $r$ is, by the previous proposition, dominated by 
\begin{align*}
\sum_{r = 0}^{p-1} (1 + \log \log T)^{r}    e^{- (1-3 \varepsilon)  (2^r V)^2/ \log \log T}
&  + \sum_{r = 0}^{p-1} (1 + \log \log T)^{r} e^{- b_{\varepsilon} (2^r V) \log (2^r V)}
\\ &  +  T^{-1/2} \sum_{r = 0}^{p-1} (1 + \log \log T)^{r},
\end{align*} 
where $b_{\varepsilon} > 0$ depends only on $\varepsilon$. 
 We can assume $V$ large, and then the exponent in the last exponential decreases by at least $b_{\varepsilon} V$ when $r$ increases by $1$, and then by more than $\log ( 2( 1+ \log \log T))$ when $T$ is large enough depending on $\varepsilon$, since $V \geq 10 \sqrt{ \log \log T}$. Hence, each term of the sum is less than half the previous one and the sum is dominated by its first term. This is absorbed in $\mathcal{O}_{\varepsilon}( e^{-c_{\varepsilon} V \log V})$.

For the last sum, we observe that $2^{p-1}  V < \log T$ by definition of $p$, and then (since we can assume $V > 1$), 
$p \ll \log \log T$, which gives a term
$$\ll T^{-1/2} ( \log \log T) (1 + \log \log T)^{ \mathcal{O} (\log \log T)} \ll T^{-1/3},$$
which can again be absorbed in  $\mathcal{O}_{\varepsilon}( e^{-c_{\varepsilon} V \log V})$ since we can assume 
$V  \ll (\log \log T)^{-1} \log T$. 

For the first sum, we separate the terms for $r \leq 10  \log \log \log T$, and for $r > 10 \log \log \log T$. 
For $T$ large, the sum of the terms for $r$ small is at most 
\begin{align*} & \sum_{r = 0}^{\lfloor 10 \log \log \log T \rfloor} ( 1 + \log \log T)^{r} 
e^{-(1-3 \varepsilon) V^2/ \log \log T} 
\\ & \ll (1 + 10 \log \log \log T) e^{ 10 \log \log \log T \log ( 1+ \log \log T)} 
e^{-(1-3 \varepsilon) V^2/ \log \log T}
\\ &   \ll 
 e^{(\log \log \log T)^3} e^{- (1- 3 \varepsilon) V^2/ \log \log T}
\end{align*}
This term is acceptable after changing the value of $\varepsilon$. 
When $r > 10 \log \log \log T$, we have (for $T$ large, $V \geq 10 \sqrt{\log \log T}$ and $\varepsilon < 1/10$)
$$(1-3\varepsilon) (2^r V)^2 / \log \log  T \geq 2^{2r}
> e^{ 20 (\log 2) \log \log \log T} 
\geq (\log \log T)^{13}.$$
The exponent is multiplied by $4$ when $r$ increases by $1$, and then decreased by more than $3 (\log \log T)^{13}$, whereas the prefactor is multiplied by $1 + \log \log T$. Hence, the term $r = \lfloor 10 \log \log \log T \rfloor + 1$
dominates the sum of all the terms $r > 10 \log \log \log T$, and its order of magnitude is acceptable.

\end{proof}
\bibliographystyle{plain}

\begin{thebibliography}{10}

\bibitem{Harper}
A.~J. Harper.
\newblock {Sharp conditional bounds for moments of the Riemann zeta function}.
\newblock {\em Preprint}, 2013.

\bibitem{HRS19}
W. Heap, M.~Radziwi\l\l~\ and K.~Soundararajan.
\newblock Sharp upper bounds for fractional moments of the {R}iemann zeta function. 
\newblock {\em The Quarterly Journal of Mathematics}, 70(4):1387--1396, 2019. 

\bibitem{HB81}
D.~R. Heath-Brown.
\newblock Fractional moments of the {R}iemann zeta function.
\newblock {\em J. London Math. Soc. (2)}, 24(1):65--78, 1981.

\bibitem{bib:KSn}
J.-P. Keating and N.~Snaith.
\newblock {Random Matrix Theory and $\zeta(1/2 + it)$ }.
\newblock {\em Commun. Math. Physics}, 214:57--89, 2000.

\bibitem{bib:Naj}
J.~{Najnudel}.
\newblock {On the extreme values of the Riemann zeta function on random
  intervals of the critical line}.
\newblock {\em Probab. Theory Relat. Fields}, 2017.

\bibitem{RS13}
M.~Radziwi\l\l~\ and K.~Soundararajan.
\newblock Continuous lower bounds for moments of zeta and {$L$}-functions.
\newblock {\em Mathematika}, 59(1):119--128, 2013.

\bibitem{Ra78}
K.~Ramachandra.
\newblock Some remarks on the mean value of the {R}iemann zeta function and
  other {D}irichlet series. {I}.
\newblock {\em Hardy-Ramanujan J.}, 1:15, 1978.

\bibitem{Ra95}
K.~Ramachandra.
\newblock {\em On the mean-value and omega-theorems for the {R}iemann
  zeta-function}, volume~85 of {\em Tata Institute of Fundamental Research
  Lectures on Mathematics and Physics}.
\newblock Published for the Tata Institute of Fundamental Research, Bombay; by
  Springer-Verlag, Berlin, 1995.

\bibitem{RS05}
Z.~Rudnick, K.~Soundararajan.
\newblock Lower bounds for moments of $L$-functions.
\newblock {\em Proc. Nat. Acad. Sci.}, 102(19):6837--6838, 2005.

\bibitem{Sound}
K.~Soundararajan.
\newblock Moments of the {R}iemann zeta function.
\newblock {\em Ann. of Math. (2)}, 170(2):981--993, 2009.

\bibitem{Tit}
E.-C. Titchmarsh.
\newblock {\em The theory of the {R}iemann zeta-function}.
\newblock The Clarendon Press, Oxford University Press, New York, second
  edition, 1986.
\newblock Edited and with a preface by D. R. Heath-Brown.

\bibitem{Tsang84}
K.-M. Tsang.
\newblock {\em The distribution of the values of the {R}iemann zeta function}.
\newblock ProQuest LLC, Ann Arbor, MI, 1984.
\newblock Thesis (Ph.D.)--Princeton University.

\bibitem{Tsang86}
K.-M. Tsang.
\newblock Some {$\Omega$}-theorems for the {R}iemann zeta-function.
\newblock {\em Acta Arith.}, 46(4):369--395, 1986.

\end{thebibliography}

\end{document}